 \documentclass[final,3p,times]{elsarticle}

\usepackage{amssymb}
\usepackage{comment}
\usepackage{color}

\usepackage{ascmac}

\usepackage{algorithm}
\usepackage{algorithmic}
\usepackage{framed}

\usepackage{amsthm,amsmath,bm}
\newtheorem{theorem}{Theorem}

\newtheorem{remark}{Remark}
\newtheorem{lemma}{Lemma}

\journal{arXiv}

\begin{document}

\begin{frontmatter}



\title{Optimal Graph Laplacian}


\author{Kazuhiro Sato}
\ead{ksato@mail.kitami-it.ac.jp}
\address{School of Regional Innovation and Social Design Engineering,
 Kitami Institute of Technology,
 Hokkaido 090-8507, Japan}

\begin{abstract}
This paper provides a construction method of the nearest graph Laplacian to a matrix identified from measurement data of graph Laplacian dynamics that include biochemical systems, synchronization systems, and multi-agent systems.
We consider the case where
 the network structure, i.e., the connection relationship of edges of a given graph, is known.
A problem of finding the nearest graph Laplacian is formulated as a convex optimization problem.
Thus, our problem can be solved using interior point methods.
However, the complexity of each iteration by interior point methods is $O(n^6)$, where $n$ is the number of nodes of the network.
That is, if $n$ is large, interior point methods cannot solve our problem within a practical time.
To resolve this issue, we propose a simple and efficient algorithm with the calculation complexity $O(n^2)$.
Simulation experiments demonstrate that our method is useful to perform data-driven modeling of graph Laplacian dynamics.
\end{abstract}

\begin{keyword}
Graph Laplacian, convex optimization, data-driven modeling 

\end{keyword}

\end{frontmatter}


\section{Introduction}
Many networked systems can be described as graph Laplacian dynamics
\begin{align}
\dot{x}(t) = -Lx(t), \label{1}
\end{align}
where $x(t)\in {\bf R}^n$ and $L\in {\bf R}^{n\times n}$ denote the state and graph Laplacian, respectively; e.g., system \eqref{1} includes biochemical systems \cite{estrada2016information, gunawardena2012linear, karp2012complex, mirzaev2013laplacian}, synchronization systems \cite{ashwin2016mathematical, dorfler2013synchronization, dorfler2014synchronization}, and multi-agent systems \cite{lns-v.96, mesbahi2010graph, olfati2004consensus}.
The graph Laplacian $L$ is determined by the connection relationship of edges of a given graph (i.e., the network structure) and weights of the edges.
In contrast to the identification of the network structure, it is difficult to identify the weights due to the lack of sensor measurements and sensor noises.
That is, it is difficult to identify the graph Laplacian $L$ using
the existing identification methods such as subspace identification methods \cite{katayama2006subspace, qin2006overview} and dynamic mode decomposition methods \cite{dawson2016characterizing, kutz2016dynamic, schmid2010dynamic}.
In other words, the existing identification methods may provide a matrix $A\in {\bf R}^{n\times n}$ that dose not have the graph Laplacian property.
For this reason, this paper proposes a construction method of the nearest graph Laplacian $L$ to a given matrix $A$ under the assumption that the network structure is known.
This assumption is motivated from the fact that it is possible to identify the network structure for many examples \cite{azuma2017structural, liu2011controllability}.

We can find related works in \cite{anderson2017distance, gillis2017computing, higham2002computing, orbandexivry2013nearest, qi2006quadratically}.
The authors of \cite{gillis2017computing, orbandexivry2013nearest} studied the problem of enforcing the stability of a system a posteriori.
That is, \cite{gillis2017computing, orbandexivry2013nearest} proposed methods for finding the nearest stable matrix to  a given unstable matrix.
Reference  
 \cite{anderson2017distance} considered the problem of finding the nearest stable Metzler matrix to a given non-Metzler matrix.
The authors of \cite{higham2002computing, qi2006quadratically} proposed methods for finding the nearest symmetric positive semidefinite matrix with unit diagonal to
a given symmetric matrix. 
However, to the best of our knowledge, there has been no previous work regarding construction of optimal graph Laplacian.

The contributions of this paper are summarized as follows.
We show that the aforementioned optimal graph Laplacian construction problem can be formulated as a convex optimization problem with an entrywise 1-norm objective function in contrast to the previous works in \cite{anderson2017distance, gillis2017computing, orbandexivry2013nearest}. In fact, the problems in \cite{anderson2017distance, gillis2017computing, orbandexivry2013nearest} are non-convex optimization problems with an entrywise 2-norm objective function due to the stability constraint, while the constraint is not required for our problem.
Thus, in contrast to \cite{anderson2017distance, gillis2017computing, orbandexivry2013nearest},
 we can obtain a global optimal solution using interior point solvers such as CVX \cite{grant2013cvx}.
However, the complexity of each iteration by interior point methods is $O(n^6)$ \cite{lin2011linearized}.
That is, if $n$ is large, interior point methods cannot solve our convex optimization problem within a practical time.
In order to solve our problem with $n\geq 1000$ within a practical time, we develop a simple and effective algorithm with the calculation complexity $O(n^2)$.
Furthermore, we demonstrate that if we replace our objective function with an entrywise 2-norm function, we cannot develop such simple algorithm.
Simulation experiments illustrate 
 that our proposed method is useful to perform data-driven modeling of graph Laplacian dynamics \eqref{1}.

{\it Notation:} The set of real numbers is denoted by ${\bf R}$.
The symbols ${\bf 0}_n \in {\bf R}^n$ and ${\bf 1}_n \in {\bf R}^n$ are column vectors with all zero entries and all one entries, respectively.
For any real number $a$, $|a|$ denotes the absolute value of $a$.
Given a matrix $A\in {\bf R}^{n\times n}$, $||A||_1$ and $||A||_2$ denote the entrywise 1-norm and 2-norm, respectively. That is, 
\begin{align*}
||A||_1 := \sum_{i=1}^n\sum_{j=1}^n|a_{ij}|, \quad
 ||A||_2 := \sqrt{\sum_{i=1}^n\sum_{j=1}^na_{ij}^2}.
\end{align*}
\section{Problem formulation} \label{sec2}

This section formulates our problem.
To this end, let $\mathcal{G} = (V,E,W)$ be a weighted graph, where $V=\{1,2,\ldots,  n\}$ is the node set, $E\subset V\times V$ is the edge set,
and $W$ is the adjacency matrix consisting of non-negative elements $w_{ij}$ called the weights.
That is, for each edge $(i,j)\in E$, the $i$-th row and $j$-th column entry of $W$ equals the weight $w_{ij}>0$, and
all other entries of $W$ are equal to zero.
The degree matrix of $\mathcal{G}$ is a diagonal matrix denoted by $D={\rm diag}(d_1,d_2,\ldots, d_n)$, with $d_i=\sum_{j=1}^n w_{ij}$.
The graph Laplacian $L$ of $\mathcal{G}$ is defined as $L:=D-W$.
Thus, the graph Laplacian $L$ has the following properties:
\begin{itemize}
\item The sum of row elements is equal to zero; i.e.,
\begin{align}
L{\bf 1}_n = {\bf 0}_n. \label{2}
\end{align}
\item The diagonal elements of $L$ are non-negative and non-diagonal elements are non-positive; i.e.,
\begin{align}
L \in S_1. \label{3}
\end{align}
\end{itemize}
Here, the set $S_1$ is defined as
\begin{align*}
S_1:= \left\{ (a_{ij})\in {\bf R}^{n\times n} | a_{ii}\geq 0,\, a_{ij}\leq 0\, (i\neq j)\right\}.
\end{align*}
Conversely, we call a matrix $L\in {\bf R}^{n\times n}$ $(n\geq 2)$ graph Laplacian if $L$ satisfies \eqref{2} and \eqref{3} \cite{lns-v.96}.
It follows from \eqref{2} that the graph Laplacian $L$ has at least one zero eigenvalue.
Furthermore, the eigenvalues of the graph Laplacian $L$ different from zero have strictly-positive real parts \cite{lns-v.96}; i.e.,
$-L$ in \eqref{1} is a stable matrix.
We assume that the network structure, i.e., the edge set $E$, is known.
This assumption comes from the fact that it is possible to identify the network structure for many examples \cite{azuma2017structural, liu2011controllability}.

The optimal solution to the following problem provides the nearest graph Laplacian to a given matrix $A$ in the case where the graph $\mathcal{G}$ is a directed graph.

\begin{framed}
Problem 1: 
\begin{align*}
&\mathop{\rm minimize}_{L\in {\bf R}^{n\times n}}\,\,\, ||A-L||_1\\
& {\rm subject\, to}\,\,\, \eqref{2}, \eqref{3},\,\, {\rm and}\,\, L\in S_2.
\end{align*}
\end{framed}

\noindent
Here, the set $S_2$ is defined as
\begin{align*}
S_2 := \left\{ (a_{ij})\in {\bf R}^{n\times n}\,|\, a_{ij}=0\,\, {\rm if}\,\, (i,j)\not\in E\,\, {\rm and}\,\, i\neq j \right\},
\end{align*}
which indicates the connection relationship of the edges of the graph $\mathcal{G}$.

Problem 1 is a convex optimization problem.
This is because the set of all $L$ satisfying constraint \eqref{2} is a vector space, i.e., a convex set, and
the sets $S_1$ and $S_2$ are also convex, and $||A-L||_1$ is a convex function \cite{boyd2004convex}.
Thus, Problem 1 can be solved by interior point solvers such as CVX \cite{grant2013cvx}.
However, the complexity of each iteration by interior point methods is $O(n^6)$ \cite{lin2011linearized}.
That is, if $n$ is large, interior point methods cannot solve Problem 1 within a practical time, as shown in Section \ref{sec5}.

Although the objective functions of \cite{anderson2017distance, gillis2017computing, orbandexivry2013nearest}  are entrywise 2-norms, 
the function of Problem 1 is an entrywise 1-norm.
This is because if we replace $||A-L||_1$ with $||A-L||_2^2$,
an algorithm for solving the modified problem becomes more complicated than the case of Problem 1, as explained in Section \ref{sec4}.

\section{Main results} \label{sec3}

This section proves the following theorem.

\begin{theorem}
Algorithm \ref{algorithm1} provides a global optimal solution to Problem 1.
\end{theorem}

\noindent
Here, $\Pi_{S_1\cap S_2}: {\bf R}^{n\times n}\rightarrow S_1\cap S_2$ in Algorithm \ref{algorithm1} denotes the projection onto the closed convex set $S_1\cap S_2$.
That is,
\begin{align*}
\left( \Pi_{S_1\cap S_2}(X)\right)_{ii} &= 
\begin{cases}
X_{ii}\quad {\rm if}\quad X_{ii} \geq 0, \\
0\quad\,\,\, {\rm if}\quad X_{ii}<0,
\end{cases}\\
\left( \Pi_{S_1\cap S_2}(X)\right)_{ij} &= 
\begin{cases}
X_{ij}\quad {\rm if}\quad X_{ij}\leq 0\quad {\rm and}\quad (i,j)\in E, \\
0\quad\,\,\,\, {\rm if}\quad X_{ij}> 0\quad {\rm or}\quad (i,j)\not\in E,
\end{cases}
\end{align*}
where $i\neq j$.
Note that Theorem 1 holds for any weighted graph $\mathcal{G}$ and any matrix $A$.

\begin{algorithm}                      
\caption{Proposed method for solving Problem 1.}         
\label{algorithm1}                          
\begin{algorithmic}[1]
\STATE $L\leftarrow \Pi_{S_1\cap S_2} (A)$.
\FOR{$i=1,2,\ldots, n$ }
\STATE $L_{ii} \leftarrow -\sum_{j\neq i} L_{ij}$.
\ENDFOR
\end{algorithmic}
\end{algorithm}

{\it Proof of Theorem 1}: Step 1 in Algorithm \ref{algorithm1} provides the global optimal solution $\tilde{L}$ to the relaxed problem of Problem 1
\begin{align*}
&\mathop{\rm minimize}_{L\in {\bf R}^{n\times n}}\,\,\, ||A-L||_1, \\
& {\rm subject\, to}\,\,\, \eqref{3}\,\, {\rm and}\,\, L\in S_2.
\end{align*}
Algorithm 1 produces $L$ with 
\begin{align}
||\tilde{L}-L||_1 = \sum_{i=1}^n |\alpha_i|, \label{key-1}
\end{align}
where $\alpha_i:= \sum_{j=1}^n \tilde{L}_{ij}$.

Because the matrix $L$ satisfies the constraint conditions of Problem 1, we prove that the matrix $L$ minimizes $||A-L||_1$.
If $\alpha_1=\alpha_2=\cdots =\alpha_n=0$, we have $L=\tilde{L}$. That is, in this case, $\tilde{L}$ is the global optimal solution to Problem 1. 
Suppose that there exists $i$ such that $\alpha_i\neq 0$, i.e., $L\neq \tilde{L}$, and $L$ is not a global optimal solution to Problem 1.
That is, a global optimal solution $L^*$ to Problem 1 satisfies
\begin{align}
||\tilde{L}-L^*||_1 < \sum_{i=1}^n |\alpha_i|. \label{key0}
\end{align}
Then, there exists $i$ such that 
\begin{align}
\sum_{j=1}^n |d_{ij}| < |\alpha_i|\quad {\rm and}\quad |\alpha_i|>0, \label{key}
\end{align}
where $d_{ij} := L^*_{ij}-\tilde{L}_{ij}$.
This is because if $\sum_{j=1}^n |d_{ij}| \geq |\alpha_i|$ for any $i$, \eqref{key0} does not hold. 
From the definitions of $\alpha_i$ and $d_{ij}$, we obtain that
\begin{align}
\sum_{j=1}^n L^*_{ij} = \alpha_i + \sum_{j=1}^n d_{ij}. \label{key2}
\end{align}
If $\alpha_i>0$, \eqref{key} and \eqref{key2} imply that
\begin{align*}
\sum_{j=1}^n L^*_{ij} \geq |\alpha_i| - \sum_{j=1}^n |d_{ij}| >0.
\end{align*}
If $\alpha_i<0$, \eqref{key} and \eqref{key2} also imply that
\begin{align*}
\sum_{j=1}^n L^*_{ij} < \alpha_i + |\alpha_i| =0.
\end{align*}
Thus, if \eqref{key0} holds, $L^*{\bf 1}_n\neq {\bf 0}_n$.
This is a contradiction that $L^*$ is a solution to Problem 1.
Hence, $L$ satisfying \eqref{key-1} and the constraint conditions of Problem 1 is a global optimal solution to Problem 1. 
This completes the proof.

The calculation complexity of Algorithm \ref{algorithm1} is $O(n^2)$.
Thus, Algorithm \ref{algorithm1} is considerably more efficient than interior point methods.
In Section \ref{sec5}, we demonstrate this fact.

\begin{remark}
Problem 1 has infinitely many global optimal solutions.
In fact, from the proof of Theorem 1, all matrices $L$ satisfying 
\begin{align*}
||\Pi_{S_1\cap S_2}(A) - L||_1 = \sum_{i=1}^n \left|\sum_{j=1}^n (\Pi_{S_1\cap S_2}(A) )_{ij} \right|
\end{align*}
are global optimal solutions to Problem 1.
\end{remark}

\section{Comments on another objective function} \label{sec4}

This section explains that if we replace the objective function $||A-L||_1$ of Problem 1 with $||A-L||^2_2$, then we cannot obtain a simple algorithm such as Algorithm \ref{algorithm1}.
To this end, we consider the case where the graph $\mathcal{G}$ is a complete graph and $A\in S_1\cap S_2$. That is, we demonstrate it using the simplest case.

We consider the following optimization problem.
\begin{framed}
Problem 2: 
\begin{align*}
&\mathop{\rm minimize}_{L_{i1},L_{i2},\ldots, L_{in}\in {\bf R}}\,\,\,  \sum_{j=1}^n (L_{ij}-A_{ij})^2, \\
& {\rm subject\, to}\,\,\, \sum_{j=1}^n L_{ij}=0,\,\, L_{ii}\geq 0,\,\, L_{ij}\leq 0 \,\,(i\neq j).  
\end{align*}
\end{framed}

\noindent
This is because under the assumption that $\mathcal{G}$ is a complete graph, the matrix $L$ constructed by using the solutions $(L_{i1},L_{i2},\ldots, L_{in})$\,\,$(i=1,2,\ldots,n)$ to Problem 2 is a global optimal solution to 
\begin{align*}
&\mathop{\rm minimize}_{L\in {\bf R}^{n\times n}}\,\,\,  ||A-L||_2^2, \\
& {\rm subject\, to}\,\,\, \eqref{2}, \eqref{3},\,\, {\rm and}\,\, L\in S_2.  
\end{align*}

\noindent
To explain that an algorithm for solving Problem 2 becomes more complicated than Algorithm \ref{algorithm1}, we need the following lemma.

\begin{lemma}
The unique global optimal solution to
\begin{align}
&\mathop{\rm minimize}_{x_1,x_2,\ldots, x_n\in {\bf R}}\,\,\, f(x_1,x_2,\ldots,x_n):= \sum_{i=1}^n (x_i-a_i)^2, \label{original}\\
& {\rm subject\, to}\,\,\, \sum_{i=1}^n x_i=0 \nonumber
\end{align}
is
\begin{align}
x_j = a_j-\frac{\sum_{k=1}^n a_k}{n} \quad (j=1,2,\ldots, n), \label{key5}
\end{align}
where $a_1,a_2,\ldots, a_n\in {\bf R}$.
\end{lemma}
\begin{proof}
The above optimization problem can be reduced to
\begin{align*}
\mathop{\rm minimize}_{x_2, x_3,\ldots, x_n\in {\bf R}}\,\,\, g(x_2, x_3,\ldots,x_n), 
\end{align*}
where 
\begin{align*}
g(x_2,x_3,\ldots,x_n) := \left( \sum_{i=2}^n x_i +a_1 \right)^2 + \sum_{i=2}^n (x_i-a_i)^2.
\end{align*}
Thus, we obtain that
\begin{align*}
\begin{pmatrix}
\frac{\partial g}{\partial x_2} \\
\frac{\partial g}{\partial x_3} \\
\vdots \\
\frac{\partial g}{\partial x_n} 
\end{pmatrix} = 2
\underbrace{
\begin{pmatrix}
2 & 1 & \cdots & 1 \\
1 & 2 & \cdots & 1 \\
\vdots & \vdots &  & \vdots \\
1 & 1 & \cdots & 2
\end{pmatrix}}_{M}
\begin{pmatrix}
x_2 \\
x_3 \\
\vdots \\
x_n
\end{pmatrix}
+2
\begin{pmatrix}
a_1 - a_2 \\
a_1 - a_3 \\
\vdots \\
a_1 - a_n
\end{pmatrix}.
\end{align*}
Because
\begin{align*}
M^{-1} = \frac{1}{n}
\begin{pmatrix}
n-1 & -1 & \cdots & -1 \\
-1 & n-1 & \cdots & -1 \\
\vdots & \vdots &  & \vdots \\
-1 & -1 & \cdots & n-1
\end{pmatrix},
\end{align*}
$\frac{\partial g}{\partial x_2} = \frac{\partial g}{\partial x_3} = \cdots = \frac{\partial g}{\partial x_n} = 0$ and $\sum_{i=1}^n x_i=0$ imply \eqref{key5}.
Because original problem \eqref{original} is a convex optimization problem, \eqref{key5} is a global optimal solution to \eqref{original}.
Furthermore, a point satisfying the Karush-Kuhn-Tucker (KKT) condition \cite{boyd2004convex} is unique.
Thus, \eqref{key5} is the unique global optimal solution to \eqref{original}.
\end{proof}

From Lemma 1, if $\sum_{j=1}^n A_{ij}>0$, then the unique global optimal solution to Problem 2 is given by
\begin{align}
L_{ij} = A_{ij} - \frac{\sum_{k=1}^n A_{ik}}{n}\quad (j=1,2,\ldots,n). \label{key6}
\end{align}
This is because \eqref{key6} implies that $L_{ii}\geq 0$ and $L_{ij}\leq 0$\,\,$(i\neq j)$.
However, if $\sum_{j=1}^n A_{ij}<0$, then \eqref{key6} does not guarantee $L_{ij}\leq 0$\,\,$(i\neq j)$.
That is, $L$ defined by \eqref{key6} is not a global optimal solution to the modified problem of Problem 1 in general.
For this reason, we cannot develop a simple algorithm such as Algorithm \ref{algorithm1} for solving the modified problem, even if
the graph $\mathcal{G}$ is a complete graph and $A\in S_1\cap S_2$.

\section{Numerical experiments} \label{sec5}

This section numerically compares Algorithm \ref{algorithm1} and CVX \cite{grant2013cvx} which is a popular solver for solving convex optimization problems.
Furthermore, we discuss eigenvalues of the graph Laplacian $L$ produced by Algorithm 1.
All computations were carried out using MATLAB R2017b on an Intel(R) Xeon(R) CPU E5-2637 v4 @ 3.50 GHz 3.50 GHz and 128 GB RAM.

We generated the matrix $A$ in Problem 1 by the following steps.
\begin{enumerate}
\item Generate the graph using the Watts and Strogats model \cite{watts1998collective}. Here, the number of the nodes is $n$.
\item Replace all nonzero elements of the adjacency matrix of the graph with $10\times {\rm rand}$, where rand is a single uniformly distributed random number in the interval $(0,1)$.
Set the modified adjacency matrix as $X$.
\item  For $i,j=1,2,\ldots, n$, 
\begin{align*}
Y_{ii} &:= \sum_{k=1}^n X_{ik}, \\
Y_{ij} &:= 0\quad (i\neq j).
\end{align*}
\item $L^*:= Y-X$.
\item $A:= L^* + s\times {\rm randn}(n)$, where $s>0$, and ${\rm randn}(n)$ denotes an $n\times n$ matrix of normally distributed random numbers.
\end{enumerate}
We here note that the Watts and Strogats model has three parameters $(n,K,\beta)$, where $K$ and $\beta$ denote the mean degree and rewiring probability, respectively.

\subsection{Comparison of Algorithm \ref{algorithm1} and CVX}

Table \ref{table1} shows the computational times (seconds) for different $n$ when $s=5$, $K=10$, and $\beta=0.3$.
When $n=1000$, $5000$, and $10000$, CVX could not solve Problem 1 due to the out of memory.
According to Table \ref{table1}, Algorithm \ref{algorithm1} is considerably faster than CVX.
In particular, Algorithm 1 could solve Problem 1 with $n=10000$ within a practical time.

\begin{table*}[t]
\caption{Comparison of Algorithm \ref{algorithm1} and CVX.} \label{table1}
  \begin{center}
    \begin{tabular}{|c|c|c|c|c|c|c|} \hline
          $n$  & 100 &  200 & 300 & 1000 & 5000 & 10000 \\ \hline 
CVX   & $5.600\times 10^{-1}$ &   $3.000\times 10^0$  & $1.004\times 10$ & out of memory & out of memory & out of memory \\ 
Algorithm \ref{algorithm1}  & $4.290\times 10^{-4}$ &  $1.340\times 10^{-3}$ & $2.496\times 10^{-3}$ & $3.267\times 10^{-2}$ & $9.464\times 10^{-1}$ & $3.926\times 10^0$  \\ \hline
    \end{tabular}
  \end{center}
\end{table*}

\subsection{Eigenvalues of $L$ generated by Algorithm \ref{algorithm1}} \label{sec5.2}

Fig.\,\ref{Fig1} illustrates eigenvalues of the matrices $L^*$, $A$, and $L$ when $n=300$, $s=5$, $K=10$, and $\beta=0.3$.
The eigenvalues of $L^*$ and $L$ were more similar than those of $L^*$ and $A$. 
That is, we could construct the graph Laplacian $L$ near the graph Laplacian $L^*$ from the matrix $A$ in the sense of the eigenvalues.
This is a preferable result if the matrix $A$ can be regarded as a perturbed matrix of the graph Laplacian $L^*$.
If this is the case, it is important that the second smallest real parts of eigenvalues of $L^*$ and $L$ are near.
This is because those determine the consensus speed of multi-agent system \eqref{1} \cite{kocarev2013consensus}.
According to Fig.\,\ref{Fig1}, those of $L^*$ and $L$ are close, although those of $L^*$ and $A$ are too different.

\begin{figure}[t]
\begin{center}
\includegraphics[width = 10cm, height = 8cm]{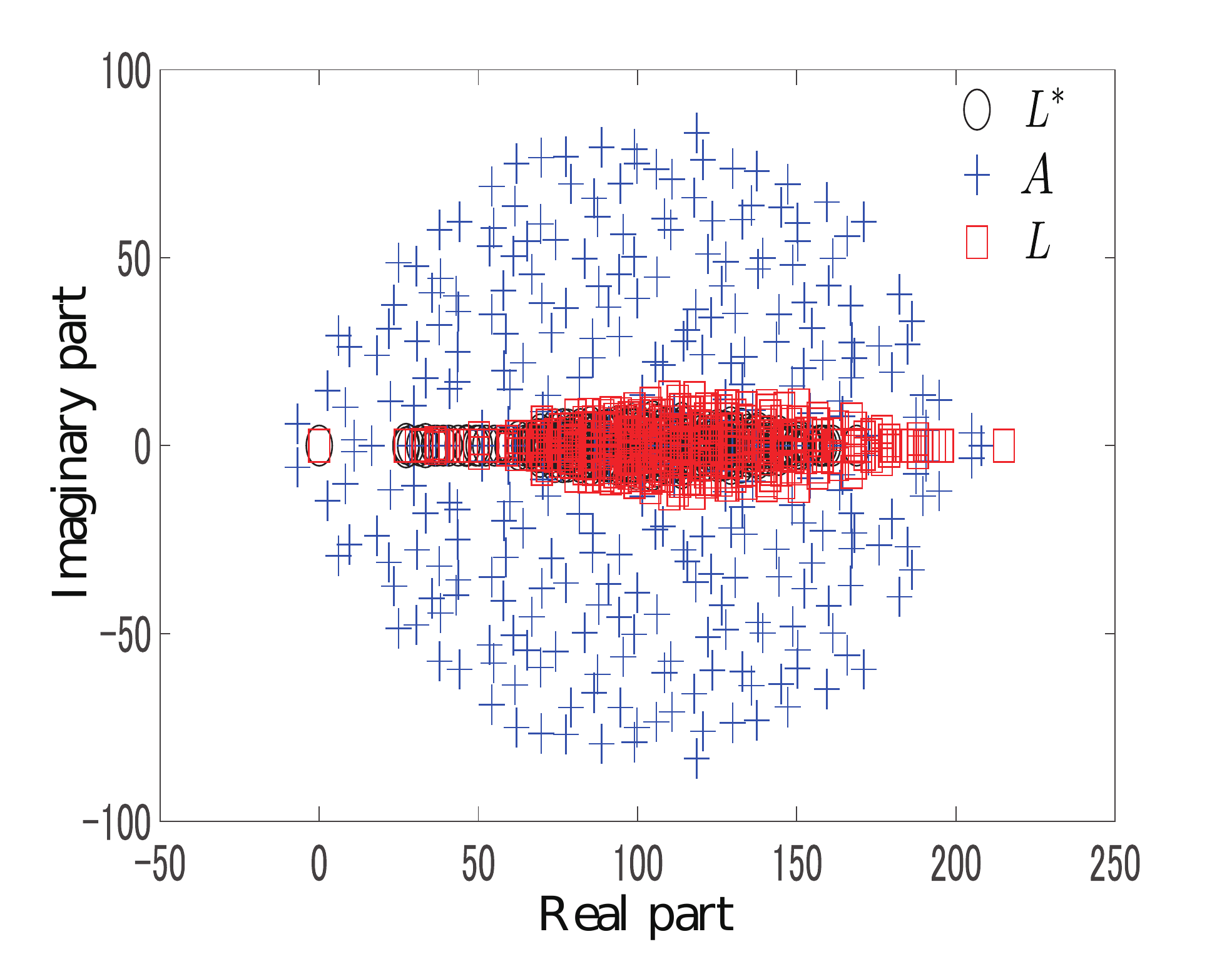}
\end{center}
\vspace{-8mm}
\caption{Eigenvalues of $L^*$, $A$, and $L$.} \label{Fig1}
\end{figure}

Moreover, Table \ref{table2} shows relations among the parameter $s$, Ave, and Var when $n=300$, $K=10$, and $\beta=0.3$,
where
\begin{align*}
{\rm Ave} &= \sum_{i=1}^{1000} \frac{|{\rm Re} (\lambda_{2,i}^*) -{\rm Re} (\lambda_{2,i})|}{1000}, \\
{\rm Var} &= \sum_{i=1}^{1000} \frac{(|{\rm Re} (\lambda_{2,i}^*) -{\rm Re} (\lambda_{2,i})| - {\rm Ave})^2}{1000}.
\end{align*}
Here, ${\rm Re} (\lambda_{2,i})$ and ${\rm Re} (\lambda_{2,i}^*)$ denote the second smallest real parts of  eigenvalues of $L$ and $L^*$, respectively, at the $i$-th trial.
Ave and Var both monotonically increase as $s$ increases.
According to Table \ref{table2}, if $0<s\leq 5$, we can expect that the second smallest real parts of eigenvalues of $L^*$ and $L$ are sufficiently close.
That is, even if each component of $A$ is relatively different from the corresponding component of $L^*$, Algorithm \ref{algorithm1} can generate $L$ close to $L^*$ in the sense of the second smallest real parts of eigenvalues of $L^*$ and $L$.

\begin{table*}[t]
\caption{Relations among $s$, Ave, and Var.} \label{table2}
  \begin{center}
    \begin{tabular}{|c|c|c|c|c|c|c|} \hline
          $s$  & 0.5 & 1 & 2&  3 & 4& 5  \\ \hline 
Ave  &0.1246  & 0.1938 & 0.4588 &   0.7743 & 1.2236 & 1.8365  \\ 
Var  &0.0107 & 0.0227 & 0.1283 &  0.3029 & 0.6844 & 1.1404  \\ \hline
    \end{tabular}
  \end{center}
\end{table*}


\subsection{Comments on an application to data-driven modeling}

The above results conclude that Algorithm \ref{algorithm1} is useful to perform data-driven modeling of graph Laplacian dynamics \eqref{1}.
In fact,
 to perform the data-driven modeling, it is desirable that
\begin{itemize}
\item we can construct the graph Laplacian $L$ from the matrix $A$ within a very short time.
\item the true graph Laplacian $L^*$ and the constructed graph Laplacian $L$  are sufficiently close. 
\end{itemize}
Table \ref{table1} indicates that Algorithm \ref{algorithm1} can produce the graph Laplacian within a very short time in contrast to CVX even if $n\approx 1000$.
Furthermore, according to Section \ref{sec5.2}, $L^*$ and $L$ are sufficiently close in the sense of 
the second smallest real parts of the eigenvalues, even if the matrix $A$ is relatively far from $L^*$.

\section{Conclusion} \label{sec6}

We have provided a simple and efficient algorithm with the calculation complexity $O(n^2)$ for solving a convex optimization problem of constructing the nearest graph Laplacian to a given matrix.
Simulation results have demonstrated that our proposed method is useful to perform data-driven modeling of graph Laplacian dynamics.

\section*{Acknowledgment}

This work was supported by JSPS KAKENHI Grant Number JP18K13773.





\bibliographystyle{model1-num-names}



\end{document}